\documentclass[12pt]{amsart}

\usepackage{amsmath}
\usepackage{amsthm}
\usepackage{amsfonts}
\usepackage{stmaryrd}
\newtheorem{proposition}{Proposition}[section]
\newtheorem{definition}[proposition]{Definition}

\newtheorem{lemma}[proposition]{Lemma}
\newtheorem{theorem}{Theorem}[section]
\newtheorem*{theorem*}{Theorem}

\newdimen\AAdi%
\newbox\AAbo%
\def\AAk#1#2{\setbox\AAbo=\hbox{#2}\AAdi=\wd\AAbo\kern#1\AAdi{}}%
\def\AAr#1#2#3{\setbox\AAbo=\hbox{#2}\AAdi=\ht\AAbo\raise#1\AAdi\hbox{#3}}%
\newcommand {\CF}{{\mathcal F}}

\newcommand {\CR}{{\mathcal R}}

\newcommand{\disp}{\displaystyle}
\newcommand{\eps}{\varepsilon}
\newcommand{\8}{\infty}

\def\m1{{-1}}

\def\S{\Sigma}
\def\s{\sigma}

\newcommand{\ul}[1]{\underline{#1}}

\newcommand{\N}{\Bbb{N}}
\newcommand{\R}{\Bbb{R}}

\theoremstyle{definition}
\newtheorem{remark}{Remark}

\usepackage{amsmath,amssymb,amsthm,epsfig,subeqnarray}
\begin{document}

\title[Renormalization: some local and global properties]
{Renormalization for  a Class of Dynamical Systems: some Local and Global Properties }
\author{Alexandre Baraviera, Renaud Leplaideur and
Artur O. Lopes}
\date{\today}

\address{Instituto de Matem\'atica, UFRGS,
91509-900 Porto Alegre, Brasil. Partially supported by PRONEX --
Sistemas Din\^amicos, Instituto do Mil\^enio, and beneficiary of
CAPES financial support.}
\address{ D\'epartement
de math\'ematiques, UMR 6205, Universite de Bretagne Occidentale,
Cedex 29238, Brest Cedex 3, France. http://www.math.univ-brest.fr/perso/renaud.leplaideur/}
\address{Instituto de Matem\'atica, UFRGS, 91509-900
Porto Alegre, Brasil. Partially supported by CNPq, PRONEX --
Sistemas Din\^amicos, Instituto do Mil\^enio, and beneficiary of
CAPES financial support.}


\maketitle
\begin{abstract}
We study the period doubling renormalization operator for dynamics which present two coupled laminar regimes with two weakly expanding fixed points. We focus our analysis on
 the potential point of view,  meaning we want to solve
$$V=\mathcal{R} (V):=V\circ f\circ h+V \circ h,$$ where $f$ and $h$ are naturally
defined. Under certain hypothesis we show the existence of a
explicit ``attracting'' fixed point $V^*$ for $\mathcal{R} $. We
call $\mathcal{R}$ the renormalization operator which acts on
potentials $V$. The log of the derivative of the main branch of the
Manneville-Pomeau map appears as a special ``attracting'' fixed
point for the local doubling period renormalization operator.

We also consider an analogous definition for the one-sided 2-full
shift $\S$ (and also for the two-sided shift) and we obtain a
similar result. Then, we consider global properties and we prove two
rigidity results. Up to some weak assumptions, we get the uniqueness
for the renormalization operator in the shift.

In the last section we show (via a certain continuous fraction
expansion) a natural relation of the two settings:  shift acting on
the Bernoulli space $\{0,1\}^\mathbb{N}$ and Manneville-Pomeau-like
map acting on an interval.

\end{abstract}

\section{Introduction}
\subsection{General presentation}
The period doubling renormalization operator was introduced by M.
Feigenbaum and by P. Coullet and C. Tresser (see \cite{CMMT} \cite{Feigenbaum1}
\cite{Feigenbaum2} \cite{Tr} \cite{TC})  for the study of
a certain class of one-dimensional dynamical systems. We recall
that for $f:[0,1]\hookleftarrow$, the renormalization of $f$ is
defined by
\begin{equation}\label{equ1-operenorm}
\mathcal{R}(f)(x)=h^{-1}\circ f^2\circ h(x),
\end{equation}
where $h$ is an affine map defined on $[0,1]$.

This operator $\mathcal{R}$  acts on dynamical transformations $f$.

A difficult problem is to find all the functions $f$ which solve the
equation $\mathcal{R}(f)=f$. In that direction, the renormalization
conjecture is that in the proper class of maps, the period doubling
renormalization operator has a unique fixed point which is
hyperbolic, with a one-dimensional unstable manifold, and  with a
codimension one stable manifold consisting of systems at the
transition to chaos (see \cite{CMMT} \cite{Co}).

The goal of this article is to present some investigations in view
to solve $\mathcal{R}(f)=f$ and to present some rigidity results.

Taking derivative in \eqref{equ1-operenorm}, and keeping in mind that $h$ is affine, we get
$$ f'(f\circ h(x))f'\circ h(x)=f'(x).$$
Then, taking the logarithm in this last equation and setting $V(x):=\log f'(x)$, we finally get
\begin{equation}\label{equ2-operenorm}
V(f(h(x)))+V(h(x))=V(x).
\end{equation}

Here we are interested in finding the solution $V$ of the above
equation.

 This way of studying the
renormalization operator is, in our view, more in the spirit of the
setting of Statistical Mechanics, where the renormalization is
looking for potentials and not looking for different dynamics (see
{\it e.g.} \cite{EFS}).

Our main motivation is chapter 5 of  the book \cite{Sch}, where the
renormalization is associated to the existence of a weakly expanding
fixed point. We look for the problems raised in \cite{Sch} but from
the point of view of potentials, not from their point of view (which
is in some sense purely  dynamical). We give rigorous mathematical
proofs. We point out that the purely dynamical problem is harder to
deal.

Unfortunately  the map $f$ still appears in \eqref{equ2-operenorm},
which is a real obstacle to solve the equation. In this article, we
are only interested by maps whose dynamics is conjugated to the full and
one-sided shift $\sigma$ acting on the Bernoulli space $\S$ with two
symbols. Indeed, one deep problem to solve the renormalization
conjecture is the huge number of different sorts of dynamical
systems one has to deal with. However, a partial solution can be
given when we assume some restrictions on the class of dynamics we
are studying. Here, we are  interested in Manneville-Pomeau like
maps: each map has 2 coupled laminar regimes, with two fixed and
weakly expanding points. These dynamics are canonically conjugated with $(\S,\s)$.

In that case, the lifted equation of
\eqref{equ2-operenorm} in $\S$ is on the form
\begin{equation}\label{equ3-operenorm}
V(\s (H(x)))+V(H(x))=V(x),
\end{equation}

where $H :\S \to \S$.
\medskip

In the case of the shift, we will consider a renormalization
operator $\mathcal{R}$ acting on potentials $V$:

given $V:\{0,1\}^\mathbb{N}\to \mathbb{R}$ in a certain natural
class of potentials $\mathcal{ F}_\gamma$, indexed by a real
parameter $\gamma$, for any $\,x:\,=\,\displaystyle
(\underbrace{0,...,0}_{c_1},\underbrace{1,...,1}_{c_2}\underbrace{0,...,0}_{c_3},1,...),$

\noindent we set

${\mathcal R}(V)(x)=
V((\underbrace{0,...,0}_{2c_1},\underbrace{1,...,1}_{c_2}\underbrace{0,...,0}_{c_3},1,...))+
V((\underbrace{0,...,0}_{2c_1+1},\underbrace{1,...,1}_{c_2}\underbrace{0,...,0}_{c_3},1,...)).$

\medskip\noindent
We show that the potential $V^*$ defined in the following way:
$V^*(x)=\gamma\, \log \frac{k+1}{k } $, for any $x$ in the cylinder
set $[\underbrace{000\dots 00}_{k} 1]$, $k \in \mathbb{N}$, is a
fixed point for $\mathcal{R}$. Moreover, for any $V\in
\mathcal{F}_\gamma$, we have that $\lim_{n\to \infty} \mathcal{ R}^n
(V) = V^*$
\medskip\noindent

One can ask if other kind of renormalization operators could be
considered (giving similar results).

\medskip\noindent
{\bf Theorem A.} {\it  For the one-side shift, there is a unique
renormalization operator (up to a constant parameter $a\in
\mathbb{N}$).}

\medskip\noindent
In this direction, we basically show that there exists a unique type
of maps $H:\S\hookleftarrow$, and a unique type of ``good''
potential $V$ which satisfy \eqref{equ3-operenorm}. From here, the
problem of solving \eqref{equ2-operenorm} lies in the study of good
projections from $\S$ onto $[0,1]$.

We also analyze a family of Manneville-Pomeau maps and a local
renormalization operator (defined in a different but similar way).

Each map of the one parameter family of Manneville-Pomeau like we
study has 2 coupled laminar regimes, with two fixed and weakly
expanding points. It is defined by
$$\left\{\begin{array}{l}
\disp f_t (x) = \frac{x}{( 1-x^t)^{1/t}} = (\frac{x^t}{1-x^t})^{1/t}, \,\mbox{ if } 0\leq x \leq \frac{1}{2^{1/t}},\\
\disp f_t(x)= (2- \frac{1}{x^t})^{1/t},\,\mbox{ if }\frac{1}{2^{1/t}}< x\leq 1,\\
\end{array}\right.$$
where  $t> 0$ is a real parameter.

Note that, for these maps, the renormalization makes sense only for
the basins of backward-attraction of the two weakly expanding fixed
points.

One important result in the setting of Manneville-Pomeau
transformation is:

\medskip\noindent
{\bf Theorem B.} {\it For each $t$, $f_{t}$ is an hyperbolic fixed
point for the doubling period renormalization operator (restricted
to each basin). The stable leaf $\CF_{t}$ is given by dynamics with
the same germs than $f_{t}$ in 0 (and in 1)}

\medskip

For each $t$ there is a unique $f_{t}$-invariant measure absolutely
continuous with respect to Lebesgue measure. In classical
(non)-uniformly hyperbolic dynamics, such a measure is referred as a
\emph{physical measure}. It is so, because one can actually ``see''
the convergence in Birkhoff averages for points. Now, it is
well-known that the nature (finite/infinite) of this invariant
measure is related and only depends on the nature of the germ of the
dynamics close to the fixed and weakly hyperbolic point (see
\cite{LSY}).

Finally, in the last section (which covers global aspects), we show
(via a certain continuous fraction expansion) a natural relation of
the two settings:  shift acting on the Bernoulli space
$\{0,1\}^\mathbb{N}$ and Manneville-Pomeau-like map acting on an
interval.

\subsection{Structure of the paper and results}

This paper is organized in the following way.

In Section \ref{sec-local} we study the local renormalization. In
Subsection \ref{subsec-pomeau} we show the fixed point property for
the renormalization operator associated to Manneville-Pomeau
transformations and also Theorem B. In Subsection
\ref{subsec-renorshift} we consider the one-sided shift and we
define there the natural renormalization operator with respect to
the class of dynamics we are considering. As a by-product we extend
the operator to the 2-side case in Subsection
\ref{subsec-renor2sideshift}, and then consider a kind of two
dimensional bijective Baker Manneville-Pomeau map in Subsection
\ref{subsec-cadobonux}.

In general terms, this section (which consider local properties) can
be considered as associated to the dynamics restricted to the  basin
of attraction (backward) of a weakly expanding fixed point. In the
shift, each different laminar regime (where the renormalization
operator acts) is associated to a different parameter $\gamma$. For
the Manneville-Pomeau map the parameter $t$ plays the role of the
$\gamma$. In Phase Transition Theory the correspondence of the two
settings (shift versus MP) is given by $\gamma=1 + \frac{1}{t}$ (in
\cite{FL} \cite{L}, the Manneville-Pomeau map is defined in a
slightly different way and indexed by a parameter $s$, the
correspondence of the parameters in the two cases, here and there,
is given by $t=s-1$).

\bigskip

In the second part of the paper, we  investigate global properties,
and get two results of rigidity.

First, we prove Theorem A in Section \ref{sec-renorshift}, that is,
there exits a unique  renormalization operator (up to an integer
positive parameter $a$) for the shift which respects the class of
dynamics we are considering (two coupled laminar regimes with two
fixed and weakly repelling points).

Several other results in the paper are for a general positive
parameter $a\in \mathbb{R}$ which appears in the late sections.

Diversity of the dynamics can therefore only follows from the choices of the laminar parameter $\gamma$ and from the choices of the conjugacies with the interval.

In Section \ref{sec-dyna-inter}  we consider a family of conjugacies
between the shift $\S$ and the interval. Instead of the laminar
parameter $\gamma$ as above,  we consider a parameter denoted by
$\alpha$. An associated value $\beta$ appears; it corresponds to a
different coupled laminar regime. For technical reasons, an extra
positive  parameter $a\in \mathbb{R}$ also appears. We then
introduce a new continuous fraction expansion and a pair of Gauss
maps associated to $\alpha$ and $\beta$.  In the last subsection we
study this family of transformation and get the second rigidity
result: up to some scaling renormalization, all these maps are
Manneville-Pomeau maps. In other words, the renormalization operator
on the shift (which is
 uniquely defined in some sense) is naturally associated, via a change
 of coordinates $\theta: \Sigma \to [0,1]$ (associated to a continuous fraction expansion), to a
 Manneville-Pomeau-like
 map (depending on certain parameters).

\bigskip

We point out that our setting has a different nature from the usual
one consider for the dynamics of one-dimensional transformations as
in \cite{CMMT} \cite{Co} \cite{Feigenbaum1} \cite{Feigenbaum2}
\cite{FW} \cite{MS} \cite{Tr} \cite{TC}.

The renormalization procedure we will consider here is associated to
the occurrence of dynamical phase-transitions (see \cite{FL},
\cite{L}, \cite{PY}, \cite{LSY}, \cite{Sa}, \cite{Sch}, \cite{Vi},
\cite{Ga} \cite{H}). Our proof do not require any of the results on
these papers. Some other mathematical references on phase
transitions are \cite{EFS} \cite{Ga} \cite{Gi} \cite{Si}.

\section{The local renormalization operator}\label{sec-local}

\subsection{The Manneville-Pomeau model}\label{subsec-pomeau}

Consider
$$\left\{\begin{array}{l}
f(x)= \frac{x}{1-x},\mbox{ if },\, 0\leq x \leq \frac{1}{2},\\
f(x)= 2- \frac{1}{x},\mbox{ if }, \, \frac{1}{2}< x\leq 1,\\
\end{array}
\right.$$

Note that one branch above is obtained from the other by the
change of coordinate $x \to (1-x).$

Consider also for $t\geq 0$,
$$\left\{\begin{array}{l}
\disp f_t (x) = \frac{x}{( 1-x^t)^{1/t}} = (\frac{x^t}{1-x^t})^{1/t}, \,\mbox{ if } 0\leq x \leq \frac{1}{2^{1/t}},\\
\disp f_t(x)= (2- \frac{1}{x^t})^{1/t},\,\mbox{ if }\frac{1}{2^{1/t}}< x\leq 1,\\
\end{array}\right.$$

For instance, in the first injective domain of
$$f_t (x) = \frac{x}{( 1-x^t)^{1/t}},$$
if $h_t (x) = x^t ,$ and, if we denote $f_1 =f$, then we have $ f_t
= h_t^{-1} \circ f \circ h_t$. Same thing for  the other injective
domain.

From this property one can get invariant a.c.i.m for each $f_t$
(just change coordinates).

In this way we have a natural partition by fundamental domains for
the branch of $f_t$ in  $(0, (1/n)^{1/t})$ by $ (\frac{1}{3^{1/t}},
\frac{1}{2^{1/t}}),..., ( \, \frac{1}{k^{1/t}}\,,\,
\frac{1}{(k+1)^{1/t}}\,),...$

For a given $y$ , the two inverse branches by $f_t$ are $\disp
x_{1,t} (y)= \frac{y}{(1+y^t)^{1/t}}$ and $x_{2,t} (y) =
(\frac{1}{2-y^t})^{1/t}$.

The image of $x_{1,t}$ is $[0\,,\,(0.5)^{1/t}]$ and   the image of
$x_{2,t}$ is $[(0.5)^{1/t}\,,\,1]$.

Note that $f_1 '(x) = \frac{1}{(1-x)^2}$ for $x \in (0,0.5)$ and
$f_1 '(x) = \frac{1}{x^2}$ for $x \in (0.5,1)$

Moreover, by the chain rule, $\disp f_t '(x) =
\frac{1}{(1-x^t)^{1+1/t}}$ for $x \in (0,0.5^{1/t})$ and $\disp f_t
'(x) = \frac{(2 x^t -1)^{-1+1/t}}{x^2}$ for $x \in (0.5^{1/t},1)$

We point out the main property of $f_t$:
\begin{equation}\label{equ1-030907}
f_t^2 (\frac{x}{2^{1/t} })\,=\,   (f_t\circ f_t)\,
(\frac{x}{2^{1/t}})= \frac{1}{ 2^{1/t}} f_t (x) .
\end{equation}

One can see by induction that
\begin{equation}\label{equ2-030907}
 f_t^j (x) = \frac{1}{ (\frac{1}{x^t} -j)^{1/t}}.
\end{equation}

\begin{definition}
For a given value $t\geq 0$ we denote ${\mathcal F}_t$ the set of
non-negative continuous functions $V:[0,1]\to \mathbb{R}$ such that
$V(x)\sim (1+ \frac{1}{t}) x^{t}$ when $x \sim 0$.
\end{definition}

Above $V(x)\sim (1+ \frac{1}{t}) x^{t}$  means $ \disp V(x)=(1+
\frac{1}{t}) x^{t}+O(x^{t + \eps}).$

\begin{definition} The renormalization operator ${\mathcal R}$ acts on the set of
functions $V$ on ${\mathcal F}_t$ by means of
$${\mathcal R} (V)\, (x) = V(\, f_t (\frac{x}{2^{1/t} }) \,)+ V(
\frac{x}{2^{1/t} }).$$
\end{definition}

We point out that our model is not just a "$\log$" version of the
one described by Schuster and Just \cite{Sch}. This is so  because
we are using here the dynamics of $f_t$, given a priori. Anyway, our
result is in the spirit of the setting of Statistical Mechanics
where the renormalization is for potentials and not for different
dynamics \cite{EFS}. In other words, we look for fixed point
potentials and not for fixed point transformations.

By recurrence and using \eqref{equ1-030907} one can easily see that
$${\mathcal R}^n (V)\, (x)\, = \,   [S_{2^n} (V)]\,(\frac{x}{2^{n/t}}) =\sum_{j=0}^{2^n} V( f_t^j (\frac{x}{2^{n/t}})).$$

Note that from \eqref{equ2-030907}
\begin{equation}\label{equ3-030907}
 {\mathcal R}^n (V)\, ( x)\,=\,  \sum_{j=0}^{2^n} V( \frac{1}{(\frac{2^{n
}}{x^t} -j)^{1/t}}).
\end{equation}

Taking derivative of both sides of \eqref{equ1-030907} one can see that
$$V^*(x) =
\log f_t\, '(x)=- ( 1 + \frac{1}{t}) \, \log (1- x^t) ,$$ is a fixed
point for ${\mathcal R}$.

Our main interest is on universality type properties
for the
renormalization operator.

\begin{theorem}For any $V\in {\mathcal F}_t$, we have that
$$ \lim_{n \to \infty} {\mathcal R}^n (V) = V^*.$$
\end{theorem}

\begin{proof}
Let $\disp x $ be in $\disp\left[\frac{1}{(m+1)^{1/t}},\frac{1
}{m^{1/t}}\right]$,  with $m\geq 2$.

Then, $\disp \frac{x}{ 2^{n/t}}$ belongs to $\disp\left[\frac{
1}{2^{n/t}(m+1)^{1/t}},\frac{1 }{2^{n/t}m^{1/t}}\right]$.

Hence, the smallest value for  $(\frac{2^{n }}{x^t} -j)$,
$j=0,1,\ldots,2^n$ is obtained when $j=2^n$, and is larger than
$\disp 2^{n}(2-1)$. Therefore each term $\disp \frac1{\frac{2^{n
}}{x^t} -j}$  is very close to $0$, and it makes sense to
approximate $V\left( f_t^j (\frac{x}{2^{n/t}})\right)$. Hence we
have

\begin{eqnarray}
{\mathcal R}^n (V)\, ( x)\,=\,  \sum_{j=0}^{2^n} V\left(
\frac{1}{(\frac{2^{n }}{x^t} -j)^{1/t}}\right)\\ = (1+ \frac{1}{t})
\sum_{j=0}^{2^n} \frac{1}{(\frac{2^{n }}{x^t}
-j)}+O(\frac{1}{(\frac{2^{n }}{x^t} -j)^{1+\eps/t}})\nonumber\\ = (1+
\frac{1}{t}) \frac{1}{ 2^{n}} \sum_{j=0}^{2^n}
\frac{1}{(\frac{1}{x^t} -\frac{j}{ 2^n})} +\frac{1}{
2^{n\eps/t}}O\left(\frac{1}{ 2^{n}} \sum_{j=0}^{2^n}
\frac{1}{(\frac{1}{x^t} -\frac{j}{
2^n})^{1+\eps/t}}\right).\label{equ1-toto1}
\end{eqnarray}

For a fixed $x$, the last expression is asymptotic to
$$  (1+ \frac{1}{t})\int_0^1\, \frac{1}{(\frac{1}{x^t} -r)}\, dr=-
(1+ \frac{1}{t})\,[\log (\frac{1}{x^t} -r)]_0^1=  -\,(1+
\frac{1}{t})\log(\frac{1}{1-x^t} ) = V^* (x).$$

\end{proof}

We now explain how this is related t and proves Theorem B.

First, note that if $V(x)=c.x^{t'}+O(x^{t'+\eps})$ with $t'<t$,
then, assuming $c>0$,
 the same proof than above yields that
$$\CR^n(V)\rightarrow+\8.$$
Therefore, only potentials $V$ in $\CF_{t}$ can converge to the
fixed point $V^*$.

Let us now assume that $V$ belongs to $\CF_{t}$. Let us set
$g:[0,1]\hookleftarrow$ such that $V=\log g'$. Then
$\CR^nV\rightarrow V^*$ is the expression of $g$ belongs to the
stable leaf of $f_{t}$ for the action of $\CR$. And as we said, this
exactly means that $g$ has the same germ than $f_{t}$.


\subsection{The one-side  shift $\Sigma$}\label{subsec-renorshift}

We consider here the Bernoulli space $\Sigma=\{0,1\}^\mathbb{N}$
and the shift acting on $\Sigma$.

We denote by $M_n \subset \Sigma, {\text{ \, for } } n \geq 1,$
the cylinder set $[\underbrace{000\dots 00}_{n} 1]$ and by $M_0$
the cylinder set $[1].$  The ordered collection
$(M_n)_{n=0}^\infty$ is a partition of $\Sigma$.

\begin{definition} Consider ${\mathcal F}$ the set of non-negative continuous functions
$V:\Sigma \to \mathbb{R}$ which are constant in the set $M_n$, for
all $n\geq 1.$ We denote by $a_n$ the value of $V$ on each $M_n$.
We further assume that $\displaystyle a_n=\frac{1}{n}+O(\frac1{n^{1+\eps}})$, for some positive $\eps$.
\end{definition}

\begin{definition} We define the renormalization operator in the
following way:

For $x:=\displaystyle
(\underbrace{0,\ldots,0}_{c_1},\underbrace{1,\ldots,1}_{c_2}\underbrace{0,\ldots,0}_{c_3},1,\ldots)$
we set
$${\mathcal R}(V)(x)=V((\underbrace{0,\ldots,0}_{2c_1},\underbrace{1,\ldots,1}_{c_2}\underbrace{0,\ldots,0}_{c_3},1,\ldots))+
$$
$$V((\underbrace{0,\ldots,0}_{2c_1+1},\underbrace{1,\ldots,1}_{c_2}\underbrace{0,\ldots,0}_{c_3},1,\ldots)).$$
\end{definition}

Note that the potential $V^*$, with value $\log\frac{k+1}{k}$ in
$M_k$, is invariant by ${\mathcal R}$. Indeed we have
$$\log\frac{k+1}{k}=\log\frac{2k+1}{2k}+\log\frac{2k+1+1}{2k+1}.$$

Gibbs states (and its decay of correlation) of the potential $V^*$,
with value $\gamma\, \log\frac{k+1}{k}$ in $M_k$,  were analyzed in
\cite{FL}\cite{L}.

\begin{theorem}
Each $V\in {\mathcal F}$ is attracted by the renormalization operator
${\mathcal R}$  to the fixed point $V^*$. \end{theorem}

\begin{proof}
 An easy computation, by induction, gives the formula
\begin{equation}\label{equ-toto}
{\mathcal R}^n(V)(x)=S_{2^n}(V)(x_n)
\end{equation}
where $\displaystyle
x_n=(\underbrace{0,\ldots,0}_{2^nc_1+2^n-1},\underbrace{1,\ldots,1}_{c_2}\underbrace{0,\ldots,0}_{c_3},1,\ldots)$
and $S_k(V)$ is the Birkhoff sum $V(.)+V\circ
\sigma(.)+\ldots+V\circ\sigma^{k-1}(.)$.

Equation (\ref{equ-toto}) yields for $x\in M_{c_1}$
\begin{eqnarray*}
{\mathcal R}(V)(x)=\sum_{j=0}^{2^n-1}a_{2^nc_1+j}&=&\sum_{j=0}^{2^n-1}\frac{1}{(2^nc_1+j)}+O\left(\frac{1}{(2^nc_1+j)^{1+\eps}}\right)\\
&=&\frac{1}{2^n}\sum_{j=0}^{2^n-1}\frac{1}{(c_1+\frac{j}{2^n})}+\frac{1}{2^{n\eps}}O(\frac{1}{2^n}\sum_{j=0}^{2^n-1}\frac{1}{(c_1+\frac{j}{2^n})^{1+\eps}})
\end{eqnarray*}
The first term in the right hand side is a Riemann sum, and
converges, as $n\to \infty$,  to
$\displaystyle\int_0^1\frac{1}{(c_1+r)}dr$. Again the second term
goes to zero.

Note that the integral $\displaystyle \int_0^1\frac{1}{(c_1+r)}dr$
is the same as  $-\log\frac{c_1+1}{c_1}$. Thus, and in the same way as before
, if the potential $V$ satisfies the condition
$$a_k=\frac{1}{k}+\frac1{k^{1+\eps}},$$
we have convergence of ${\mathcal R}^n(V)(x)$ to $V^*(x)$ when $n$
goes to $+\infty$.

\end{proof}

\begin{remark}\label{rem-gamma}
 A similar result can be obtained for  $V^*(x) = \gamma\,
\log\frac{k+1}{k}=a_k$, when $x \in M_k$, and $\gamma>1$ is fixed.
In this case we have to consider  ${\mathcal R}$ acting on the set
${\mathcal F}_\gamma$, as the set of $V$ such that $a_k=\disp
\frac{\gamma}{k}+O(\frac\gamma{k^{(1+\eps)}})$.

\end{remark}
\subsection{The two-sided  shift $\hat{\Sigma}$}\label{subsec-renor2sideshift}

We denote $\hat{\Sigma}= \{0,1\}^\mathbb{Z}$ and also denote each
point in this set by $<y|\,x>= <...y_2,y_1|\, x_0,x_1,x_2..>$ where
$x$ is future and $y$ is past. The shift $\hat{\sigma}$ is defined
by
$$ \hat{\sigma}( <...y_2,y_1|\, x_0,x_1,x_2..>)=
<...y_2,y_1,x_0|\,x_1,x_2..>.$$

\begin{definition} Consider ${\mathcal F}$ the set of non-negative continuous functions
$V:\Sigma \to \mathbb{R}$, which are constant in the sets of the
form
$$M_m|M_n=\{<y,x>, x \in M_n, y \in M_m\},$$
for each pair $m,n\geq 1.$ We denote by $a_{m,n}=V(m,n)$ the value
of $V$ on each $M_m \times M_n$. We further assume that
$\displaystyle a_{m,n}=\frac{m+n}{(m-1)\, n}+O\left(\frac1{m^{1+\eps_{1}}}\right)+O\left(\frac1{n^{1+\eps_{2}}}\right)$, for positive $\eps_{i}$.
\end{definition}

\begin{definition} We define the renormalization operator in the
following way:

For
$$z:=\displaystyle
(\underbrace{0,\ldots,0}_{d_3},\underbrace{1,\ldots,1}_{d_2}\underbrace{0,\ldots,0}_{\zeta})
\,|\,
(\underbrace{0,\ldots,0}_{c_1},\underbrace{1,\ldots,1}_{c_2}\underbrace{0,\ldots,0}_{c_3},1,\ldots)
\in M_\zeta \times M_{c_1},$$ we set
$${\mathcal R}(V)(z)=V(\underbrace{0,\ldots,0}_{d_3},\underbrace{1,\ldots,1}_{d_2}\underbrace{0,\ldots,0}_{2\,\zeta-1})
\,|\, (\underbrace{0,\ldots,0}_{2\, c_1+
1},\underbrace{1,\ldots,1}_{c_2}\underbrace{0,\ldots,0}_{c_3},1,\ldots))+
$$
$$V(\underbrace{0,\ldots,0}_{d_3},\underbrace{1,\ldots,1}_{d_2}\underbrace{0,\ldots,0}_{2\, \zeta})
\,|\, (\underbrace{0,\ldots,0}_{2\,
c_1},\underbrace{1,\ldots,1}_{c_2}\underbrace{0,\ldots,0}_{c_3},1,\ldots)).$$
\end{definition}

In order to simplify the notation we write
$${\mathcal R}(V)(z) = V(2 \zeta -1,2 c_1 +1) + V(2\, \zeta, 2
c_1).$$

One can show that for $V \in {\mathcal F}$, and $z \in M_m|M_n$, we
have that
$${\mathcal R}^n (V)(z) = \sum_{k=0}^{2^n-1} V(2^n \zeta - 2^n +1 +k,
2^n c_1 + 2^n -1-k).$$

It is easy to see that the potential given by: for each $z \in
M_j|M_k$
$$V^* (z) = \log \, \frac{j\, (k+1)}{(j-1) \, k},$$
defines a fixed point potential for ${\mathcal R}$.

\begin{theorem}
Each $V\in {\mathcal F}$ is attracted by the renormalization operator
${\mathcal R}$  to the fixed point $V^*$. \end{theorem}

\begin{proof}
 Given  $V\in {\mathcal F}$, we have
\begin{eqnarray*}
{\mathcal R}^n (V)(z) &=& \sum_{k=0}^{2^n-1} V(2^n \zeta - 2^n +1
+k,
2^n c_1 + 2^n -1-k)\\
&=&\sum_{k=0}^{2^n-1}\left( \frac{2^n (c_1 +\zeta)}{ (2^n \zeta -
2^n +k)\,(2^n c_1 + 2^n
-1-k)}+O\left(\frac1{(2^n \zeta - 2^n +k+1)^{1+\eps_{1}}}\right)\right.\\
&&\hskip 3cm +\left.O\left(\frac1{(2^n c_1 + 2^n
-1-k)^{1+\eps_{2}}}\right)\right)\\
&=&\frac{1}{2^n} \,\sum_{k=0}^{2^n-1} \frac{(c_1 +\zeta)}{ ((\zeta -
1) +\frac{k}{2^n}\,)\,\,(\,(c_1 + 1) - \frac{k+1}{2^n}
)}+O(\frac1{2^{n\eps_{1}}})+O(\frac1{2^{n\eps_{2}}})
\end{eqnarray*}

Taking $n$ large we get
$$ \frac{1}{2^n} \,\sum_{k=0}^{2^n-1} \frac{(c_1 +\zeta)}{ ((\zeta - 1) +\frac{k}{2^n}\,)\,\,(\,(c_1 +
1) - \frac{k+1}{2^n}  )}\sim $$
$$ \int_0^1\, \frac{(c_1 +\zeta)}{ ((\zeta - 1) +x\,)\,\,(\,(c_1 +
1) - x  )}\, dx=$$
$$ (c_1 +\zeta)\, \int_0^1\, \frac{1}{ (\zeta - 1 +x\,)}\,+\,\frac{1}{ \,(c_1 +
1 - x ) }\, dx=\left[\,\log \frac{\zeta-1 +x}{c_1
+1-x}\,\right]_0^1=$$
$$ \log\,(\, \frac{\zeta\, (c_1 +1)}{(\zeta-1) \, c_1}\,)=V^*(z).$$
\end{proof}

\subsection{The Baker Manneville-Pomeau bijective
transformation}\label{subsec-cadobonux}

Using the notation of the first section, for a fixed value of $t$,
consider $$F_t: [0,1]\times [0,1] \to  [0,1]\times [0,1] ,$$ a
bijective transformation such that satisfies for each $x $ and $y$
$$ F_t(x, f_t(y))= (f_t(x), y).$$

\begin{definition}
For a given value $t\geq 0$, we denote ${\mathcal F}_t$ the set of
non-negative continuous functions $V:[0,1]\times [0,1]\to
\mathbb{R}$ such that $V(x,y)=  \, (1+ \frac{1}{t})\,
\log(\frac{1+x}{1-y})+O(x^{1+\eps_{2}})+O(y^{1+\eps_{1}})$ when
$(x,y) \sim (0,0)$.
\end{definition}

In order to simplify the notation we consider here only the case
$t=1$. Similar results will be true for the general case $t\geq 0.$
We use the notation ${\mathcal F}_1={\mathcal F}.$

\begin{definition} The renormalization operator ${\mathcal R}$ acts on the set of
functions $V$ on ${\mathcal F}$ by means of
$${\mathcal R} (V)\, (x,y) = V(\frac{2}{2+x }, \frac{y}{2-y})+ V(
\frac{x}{2 }, \frac{y}{2 })$$
\end{definition}

In the same way as before one can show that
$$ V^* (x,y) = \log( \frac{1+x}{1-y}),$$
is a fixed point for ${\mathcal R}$.

We leave for the reader the proof of the theorem:

\begin{theorem}
Each $V\in {\mathcal F}$ is attracted by the renormalization operator
${\mathcal R}$  to the fixed point $V^*$.
\end{theorem}


\section{Uniqueness of the renormalization on the shift}\label{sec-renorshift}

In this section, we consider global properties and we prove that, up
to relatively weak assumptions, there exists a unique type of renormalization
operator in the shift.

We first state a simple lemma:
\begin{lemma}
Let $(X_{1},T_{1})$ and $(X_{2},T_{2})$ be two conjugated dynamical systems. Let $\theta:X_{1}\rightarrow X_{2}$ be the conjugacy.
If $H_{1}$ satisfies
$\disp H_{1}^{-1}\circ T_{1}^2\circ H_{1}=T_{1}$,
then $H_{2}:=\theta\circ H_{1}\circ  \theta^{-1}$ satisfies
$$H_{2}^{-1}\circ T_{2}^2\circ H_{2}=T_{2}.$$
Moreover if $V_{1}$ satisfies $V_{1}(T_{1}(H_{1}(x)))+V_{1}(H_{1}(x))=V_{1}(x)$, then $V_{2}:=V_{1}\circ \theta^{-1}$ satisfies
$$V_{2}(T_{2}(H_{2}(x)))+V_{2}(H_{2}(x))=V_{2}(x).$$
\end{lemma}

In view of those results, it's meaningful to study maps $H$ on the
shift  which satisfy the property $H^{-1}\circ\s^2\circ H=\s$. More
assumptions are necessary, if one wants to respect some other
properties of the map $x\mapsto x/2$ in the interval. If $\bar0^\8$
in the shift represents the 0 of the interval, then
$H(\ul0^\8)=\ul0^\8$ needs to hold. Moreover the map $H$ has to
``increase'', which can be translated into ``$H$ respect the
lexicographic order in $\S$''. The uniqueness of the type of such
map $H$  follows from the next proposition:

\begin{proposition}\label{prop-renor-shift}
Let $H$ be an increasing function on the shift $\S_{2}$ (for the
lexicographic order), such that
\begin{enumerate}
  \item for every $\ul{x}=(1,x_{2},x_{3},\ldots)$, $\disp H(\ul x)=(\underbrace{0,\ldots,0}_{a \ terms},1,x_{2},x_{3},\ldots)$, where $a\geq 1$;
  \item $H^{-1}\circ\s^2\circ H=\s$,
  \item $H(\ul 0^\8)=\ul 0^\8$
\end{enumerate}
Then, for every $\disp \ul x=(\underbrace{0,\ldots,0}_{n_{0}\
terms},1,x_{n_{0}+2}\ldots)$, we have $\disp
H(x)=(\underbrace{0,\ldots,0}_{2n_{0}+a\
terms},1,x_{n_{0}+2},\ldots)$.
\end{proposition}
\begin{proof}
Note that the formula is correct for every $\ul x$ on the form
$(1,\ldots)$. We first consider the case where $a\ge 2$. Note that
we took $a=1$ in a previous section where we considered the shift.

Let us pick some $x$, which necessarily has to be of the form $\disp
\ul x=(\underbrace{0,\ldots,0}_{n_{0}\ terms},1,x_{n_{0}+2}\ldots)$.
We assume $n_{0}>1$. We point out that $\s(\ul x)\ge \ul x$, because
a ``1'' appears sooner in $\s(\ul x)$ than in $\ul x$. Therefore we
must have
\begin{equation}\label{equ1-ordre-shift}
H(\s(\ul x))> H(\ul x), \mbox{ if }x\neq \ul 0^\8,\ul 1^\8.
\end{equation}
Now, $\s^{n_{0}}(\ul x)$ belongs to the cylinder $[1]$, hence
$H(\s^{n_{0}}(\ul x))=[\ul a\ul x]$, where $\ul a$ is the finite
word $\underbrace{0,\ldots,0}_{a\ terms}$, and $[\quad ]$ is the
concatenation of words in the shift. As we said before, in the
moment we are considering such $a\geq 2$. The constraint $\disp
H^{-1}\circ\s^2\circ H=\s$, yields $\s^{2n_{0}}\circ H=H\circ
\s^{n_{0}}$. Therefore
\begin{equation}\label{equ2-ordre-shift}
H(\ul x)=(\underbrace{?,\ldots,?}_{2n_{0}\ terms},\underbrace{0,\ldots,0}_{a\ terms},1,x_{n_{0}+2},\ldots),
\end{equation}
where the first $2n_{0}$ digits are unknown.

As $H$ has the increasing property, its image is in the cylinder
$[0]$, and the first digit in \eqref{equ2-ordre-shift} is 0. The
property $\disp H^{-1}\circ\s^2\circ H=\s$, also means $\disp
\s^2\circ H=H\circ\s$. Therefore, each odd unknown digit in
\eqref{equ2-ordre-shift} is 0.

Now, we prove that no even unknown digit can be 1. Let us assume that the second digit is 1. Doing the same work for $\s(\ul x)$
(here we use $n_{0}>1$), we have
\begin{equation}\label{equ3-ordre-shift}
H\circ\s(\ul x)=(\underbrace{0,?,\ldots,0,?}_{2n_{0}-2\ terms},\underbrace{0,\ldots,0}_{a\ terms},1,x_{n_{0}+2},\ldots),
\end{equation}
where each unknown digit at position $2p$ is the same digit than the digit in position $2p+2$ in \eqref{equ2-ordre-shift}.
To get these equalities, we again used $\disp \s^2\circ H=H\circ\s$.

If the second digit in \eqref{equ2-ordre-shift} is a ``1'', then to
respect \eqref{equ1-ordre-shift}, the second digit in
\eqref{equ3-ordre-shift} must be a ``1'' too. Therefore, the cascade
rule yields that each even unknown digit must be 1, in
\eqref{equ2-ordre-shift} and in \eqref{equ3-ordre-shift}. In that
case, and as we assume $a\ge 2$, there will be a ``1'' in $H(\ul x)$
in position $2n_{0}$, and a ``0'' for $H\circ\s(\ul x)$, and the two
words coincide before that position. Hence, $H(\s(\ul x))<H(\ul x)$,
which is impossible by \eqref{equ1-ordre-shift}. This proves that
the assumption is false, and the second unknown digit in
\eqref{equ2-ordre-shift} must be a ``0''.

Note that this also holds if $n_{0}=1$. Indeed, in that case we
completely know $H\circ\s(\ul x)$, by assumption {\it (1)} in the
proposition. Therefore the above discussion means that for every
$\ul \xi=(0,\ldots)$, $H(\ul\xi)$ starts with 3 ``0''.  Here again,
the cascade rule between \eqref{equ2-ordre-shift} and
\eqref{equ3-ordre-shift} yields that every even unknown digit is
``0''.

To complete the proof of the proposition, we have to deal with the case $a=1$. In that case, the assumption
``the second unknown digit in \eqref{equ2-ordre-shift} in 1'' yields to
\begin{eqnarray*}
H(\ul x)&=&(\underbrace{0,1,\ldots,0,1,0,1}_{2n_{0}\ terms},\stackrel{{\downarrow a}}{0},1,x_{n_{0}+2},\ldots),\\
H\circ\s(\ul x)&=&(\underbrace{0,1,\ldots,0,1}_{2n_{0}-2\ terms},\stackrel{{\downarrow a}}{0},1,x_{n_{0}+2},\ldots).
\end{eqnarray*}
Hence, the unique possibility to respect the increasing property for
$H$ would be to  alternate  ``0'' and ``1'' for the tail of $\ul x$.
But even in that case, this will be in contradiction with
\eqref{equ1-ordre-shift}. This finishes the proof.
\end{proof}

The conclusion is that each renormalization operator has to be of
the form: take a fixed $a\in \mathbb{N}$, then given
$V:\{0,1\}^\mathbb{N}\to \mathbb{R}$, for any
$\,x:\,=\,\displaystyle
(\underbrace{0,...,0}_{c_1},\underbrace{1,...,1}_{c_2}\underbrace{0,...,0}_{c_3},1,...),$

\noindent we set

${\mathcal R}(V)(x)=
V((\underbrace{0,...,0}_{2c_1},\underbrace{1,...,1}_{c_2}\underbrace{0,...,0}_{c_3},1,...))+
V((\underbrace{0,...,0}_{2c_1+a},\underbrace{1,...,1}_{c_2}\underbrace{0,...,0}_{c_3},1,...)).$

\section{Dynamics on the interval}\label{sec-dyna-inter}

\subsection{Choices of the laminar parameter}
From expression (9) in \cite{CLR}
$$x=\disp\frac1{\disp (c_{1}+1 )\,+\frac1{\disp c_{2}+\frac1{\disp c_{3}+...}}}\,=\,\theta ( \bar{x} ),$$
one gets a change of coordinates  $\theta$ from the shift (where a
point $\bar{x}$ in  $\Sigma$ is denoted by $\bar{x}= \displaystyle
(\underbrace{0,\ldots,0}_{c_1},\underbrace{1,\ldots,1}_{c_2},\underbrace{0,\ldots,0}_{c_3},1,\ldots)\in\{0,1\}^\mathbb{N}$)
to the interval $[0,1]$. This map preserves the lexicographic order.

 Since we  consider in the first place the
 potential, and not the dynamics,
 we point out that each $c_{i}$, for $i\ge 1$,
and $c_{1}+1$ can be replaced by
$$\frac1{(1+\frac1n)^\alpha-1},$$
where $n=c_{1}+1,c_{2},c_{3},\ldots$ and $\alpha=1$.

Now, remembering Remark \ref{rem-gamma}, we  note that each term
$\disp (1+\frac1{c_{i}})^\alpha$ is some $e^{V^*(z_{i})}$, where
$z_{i}$ is in relation with the orbit of the initial point under the
actions of the two coupled and competitive laminar regimes.

The goal in this Section \ref{sec-dyna-inter} is to obtain a more
general change of coordinates of this sort, in order to take care of
the choice of different possibilities of the parameter $\alpha\geq
1$.

\subsection{Convergence of a new continued fraction expansion}
In this section we define a new type of continued fraction.

Let $\alpha>0$ be a real number. We define $g:(0,\infty) \to
\mathbb{R}$, given by
$$g_\alpha (z)= g(z) = \frac{1}{(1 + \frac{1}{z})^\alpha -1}.$$
For a fixed $\alpha$, and when the meaning  is clear, we omit the
subscribe $\alpha$ in $g_{\alpha}$, in  order to make the formulas
simpler.

We have for every $z\in(0,+\8)$, $\disp
g'(z)=\frac\alpha{z^2}\frac{1}{((1 + \frac{1}{z})^\alpha -1)^2}(1 +
\frac{1}{z})^{\alpha-1}$, hence $g$ is increasing. Moreover,
$\lim_{z\to 0 } g(z)= 0$ and  $\lim_{z\to +\8 } g(z)=+\8$.
Therefore, for any given $y\in(0,\infty)$, there exists a $n_y \in
\mathbb{N}$, such that
\begin{equation}\label{equ1-gauss-alpha}
g(n_y) \leq y < g(n_y+1).
\end{equation}

 Moreover, $\disp g(z)= z^\alpha+o(z^\alpha)$ when $z$ is close to 0, and,
 $\disp g(z)=\frac{z}\alpha-\frac{\alpha-1}{2}+O(\frac1z)$, when $z$ is  close to $+\8$.

\begin{lemma} The map $g_{\alpha}$ is convex for $\alpha>1$, and concave for
$\alpha<1$.\end{lemma}
\begin{proof}
To prove this scholium, first note that $\disp
g'(z)=\frac{\alpha}{z^2+z}\left(g(z)+g^2(z)\right)$. This yields
$$g''(z)=-\alpha\frac{2z+1}{(z^2+z)^2}\left(g(z)+g^2(z)\right)+\frac\alpha{z^2+z}\left(g'(z)+2g'(z)g(z)\right).$$
If we replace in this last expression the value of $g'(z)$ in function of $z$ and $g(z)$, we get
$$g''(z)=\alpha^2\frac{g(z)+g^2(z)}{(z^2+z)^2}\left(g(z)-\left(\frac{z}{\alpha}-\frac{\alpha-1}2\right)\right).$$
Note that $\frac{z}{\alpha}-\frac{\alpha-1}2$ is the asymptote of
$g$ close to $+\8$. Then, the convexity of the map depends on the
position of the graph with respect to the asymptote. It's convex
when the graph is above the asymptote, and it's concave when the
graph is below the asymptote. Now, recall that a convex map has a
non-decreasing derivative, and a concave map has a non-increasing
derivative. Therefore, easy considerations on the relative position
of the graph with respect to the asymptote prove that the graph
cannot cross the asymptote. Hence the map is convex for $\alpha>1$,
and concave for $\alpha<1$.
\end{proof}

Note that  $g_{\alpha}(1)= \frac{1}{2^\alpha -1}$. Therefore,
$g_{\alpha}(1)<1$, for $\alpha>1$, and $g_{\beta}(1)>1$, for
$\beta<1$.
\begin{lemma}\label{lem-cont-frac}
Let $(a_k)_{k\in\N}$ be a sequence of real numbers such that
$a_{0}=0$, each $a_{2k+1}$ is larger than 1, and all the even terms
$a_{2k}$, $k>0$, are positive and uniformly bounded away from zero.
Then,
the sequence of real numbers $(r_{k})$ defined by
$$r_{k}=\disp\frac1{\disp a_{1}+\frac1{\disp a_{2}+\frac1{\ddots+\disp\frac1{\disp a_{k}}}}},$$
converges to a real number denoted by $[0,a_{1},a_{2},a_{3},\ldots]$, and we have
$$[0,a_{1},a_{2},a_{3},\ldots]= \disp\frac1{\disp a_{1}+\frac1{\disp a_{2}+\frac1{\ddots+\disp\frac1{\disp a_{k}+\frac1{\ddots}}}}}.$$
\end{lemma}
\begin{proof}
Let $(a_k)_{k\in\N}$ be as in the assumptions. We define two new sequences  $(p_k)_{k\in\N}$
and $(q_k)_{k\in\N}$, by induction:
$$p_0=0,\ p_1=1,\ q_0=1,\ q_1=a_1$$
$$\forall k\in\N,\ p_{k+2}=a_{k+2}p_{k+1}+p_k,\ q_{k+2}=a_{k+2}q_{k+1}+q_k.$$
It's easy to see, by induction, that for every $k>0$, $q_{k}\ge1$.
Using $a_{2k+1}\ge 1$, we easily get  $q_{2k+1}\ge k$, and then
$q_{2k}\ge A.k$, where $A$ is a positive lower bound for all the
$a_{2j}$'s. Therefore,  $q_{k}$ goes to $+\8$ as $k$ increases to
$+\8$.

If we set  $u_k=p_{k+1}q_{k}-p_{k}q_{k+1}$, then $u_{k+1}=-u_{k}$
for every $k$. We claim that $r_k=\disp\frac{p_k}{q_k}$. Then, the
two subsequences $(r_{2k})$ and $(r_{2k+1})$ are mutually adjacent
and converge to the same limit. We leave to the reader to check that
the even sequence $(r_{2k})$ increases and the odd sequence
$(r_{2k+1})$ decreases.
\end{proof}

 We now consider real numbers, $\alpha$ in $[1,+\8[$, $\beta$ in $]0,1]$,  and  the natural number $a\geq 0$.  Given $\bar{x}=  \displaystyle
(\underbrace{0,\ldots,0}_{n_0},\underbrace{1,\ldots,1}_{n_1},\underbrace{0,\ldots,0}_{n_2},1,\ldots)\in \Sigma=\{0,1\}^\mathbb{N}$
we define a real number in $[0,2^\beta-1]$ in the following way:

{\footnotesize
$$ \theta_{\alpha,\beta,a} (\bar{x})= \displaystyle{\frac{      1}{ \disp   \frac{(n_0+1)^\beta}{(n_0+2)^\beta - (n_0+1)^\beta}
+ \frac{1}{ \disp\frac{(n_1+a)^\alpha}{(n_1+a+1)^\alpha -
(n_1+a)^\alpha} + \frac{1}{\disp \frac{n_2^\beta}{(n_2+1)^\beta -
n_2^\beta} +       \frac{1}{\disp
\frac{(n_3+a)^\alpha}{(n_3+a+1)^\alpha - (n_3+a)^\alpha } +... } }
}}}
$$}
We effectively claim (and let the reader check) that the sequence
defined by $a_{2k}=g_{\alpha}(n_{2k-1}+a)$ and
$a_{2k+1}=g_{\beta}(n_{2k})$ satisfies the properties of Lemma
\ref{lem-cont-frac}. Therefore the real number
$[0,a_{1},a_{2},\ldots]$ is well-defined.

We now claim that $\theta_{\alpha,\beta,a}(\bar x)$ belongs to
$[0,2^\beta-1]$. Indeed, the odd subsequence $(r_{2k+1})$ decreases
and the even subsequence $(r_{2k})$ increases. To minimize the value
of $\theta_{\alpha,\beta,a}(\bar x)$, it is necessary and sufficient
to maximize $n_{0}$. On the other hand, to maximize the value of
$\theta_{\alpha,\beta,a}(\bar x)$, it's necessary and sufficient to
minimize $n_{0}$ and to maximize $n_{1}$. Therefore, for every $\bar
x$,
$$0=\theta_{\alpha,\beta,a}(\bar 0^\8)\leq \theta_{\alpha,\beta,a}(\bar x)\leq \theta_{\alpha,\beta,a}(\bar 1^\8)=2^\beta-1.$$

\begin{remark}
The number $a$ does not need to be in $\N$ to define
$\theta_{\alpha,\beta,a}$, but in $\R^+$. This restriction will be
explained later. Note also that $\theta$ is not an homeomorphism.
\end{remark}

\subsection{Identification of the Image Interval}
In this subsection, we study a one-parameter family of Gauss-like transformations. The goal is to prove that for good parameters, the image $\theta_{\alpha,\beta,a}(\S)$ is the interval $[0,2^\beta-1]$.

\paragraph{{\em The case $\alpha>1$}}
Convexity  and the existence of the asymptote yield, for every $n\ge
1$,

\begin{equation}\label{equ1-alpha-zob}
0<\, g_{\alpha}(n+1) - g_{\alpha}(n) \,<\,
\sup_{j>0}\left\{g_{\alpha}(j+1) - g_{\alpha}(j)\right\}=1/\alpha.
\end{equation}

For a positive  $y$, we set $\disp 0\leq r_{\alpha}(y) = y -
g(n_y)\leq \frac{1}{\alpha}$.

If $\alpha=1$, then $2^\alpha-1=1$.  In this case,  $g(z)=z$ (for
all $z$). Therefore, $g(n)=n$, and $r(x)= \frac{1}{x} - [
\frac{1}{x} ]$ is the usual fractional  part of $1/x$.

For the  case, $\alpha > 1$, we  consider the new Gauss-like map $
\phi_\alpha :(0,2^\alpha-1)\to [0,1]$ given by
$$\phi_\alpha (x) =     \frac{1}{x}   - g_{\alpha}\left(\left[ \frac{1}{(x+1)^\frac{1}{\alpha} -1} \right]\right).$$

\paragraph{{\em The case $\beta<1$}}
We set
$$\phi_\beta (x) =     \frac{1}{x}   - g_{\beta}\left(\left[ \frac{1}{(x+1)^\frac{1}{\beta} -1} \right]\right).$$
Note that concavity and existence of the asymptote yield, for every
$n\geq 1$,

\begin{equation}\label{equ2-alpha-zob}
\frac1\beta<\, g_{\alpha}(n+1) - g_{\alpha}(n) \,\le\,
\sup_{j>0}\left\{g_{\alpha}(j+1) -
g_{\alpha}(j)\right\}=g_{\beta}(2)-g_{\beta}(1).
\end{equation}

We now assume that $\alpha$, $\beta$ and $a$, satisfy
 \begin{subeqnarray}\label{toto}
\frac{1}{\left(\frac32\right)^\beta-1}-\frac1{2^\beta-1}&=&(1+\frac1{a+1})^\alpha-1.\slabel{equ1-alpha-beta}\\
\frac1\alpha&=&2^\beta-1.\slabel{equ1.2-alpha-beta}
\end{subeqnarray}

This system of conditions is referred as \eqref{toto}. Note that
this yields
 \begin{equation}
\label{equ2-alpha-beta}
\frac1{\disp\frac{1}{\left(\frac32\right)^\beta-1}+0}=\frac1{\disp\frac1{2^\beta-1}+\frac1{\disp\frac1{(1+\frac1{a+1})^\alpha-1}+0}}.
\end{equation}

Note also that if $\alpha=1$, then $\beta=1$. In this case $a=0$.


We first check that conditions \eqref{toto}  are compatible with our
assumptions $\alpha\ge 1$ and $\beta\le 1$.

Note that $\beta\le1$ yields $2^\beta-1\le1$, and, then, we indeed
have $\alpha\ge 1$.

One can ask: which  values $a\in \mathbb{R}$ are possible?

Solving in $a$ (the two equations) as a function of $\beta$, we have to
consider the map
$$\disp \beta \mapsto
a(\beta)+1:=\frac1{\left(\frac1{\frac1{(3/2)^\beta-1}-\frac1{2^\beta-1}}+1\right)^{2^\beta-1}-1}.$$

\begin{lemma} The map

$$x\,\to\,
\frac1{\left(\frac1{\frac1{(3/2)^x-1}-\frac1{2^x-1}}+1\right)^{2^x-1}-1}$$

is  a decreasing bijection
from $]0,1]$ onto $[1,+\8]$.

\end{lemma}

\begin{proof} Once $(2^x -1)$ is increasing, and
$\frac1{\frac1{(3/2)^x-1}-\frac1{2^x-1}}+1$ is bigger than $1$; all
we have to prove is  that
$\frac1{\frac1{(3/2)^x-1}-\frac1{2^x-1}}+1$ is also
increasing.

Given $f:[0,1]\to \mathbb{R}$ and $f:[0,1]\to \mathbb{R}$, one can
ask when

$\frac{1}{f(x)} - \frac{1}{g(x)}$ is decreasing? Taking
derivative we get the condition
$$\frac{f'(x)}{f(x)^2}\,  \,>\, \frac{g'(x)}{g(x)^2}\, $$

Suppose $f(x)= (\, (3/2)^x -1)$, then the first term is
$$ \frac{ \log (3/2) \, (3/2)^x}{\, ((3/2)^x -1)^2}.$$

Suppose $g(x)= (\, 2^x -1)$, then the second term is
$$ \frac{ \log 2 \, 2^x}{\, (2^x -1)^2}.$$

We claim that, for all $x\in [0,1]$,
$$ \frac{ \log (3/2) \, }{\, ((3/2)^x -1)\, (1- (3/2)^{-x})}= \frac{ \log (3/2) \, (3/2)^x}{\, ((3/2)^x -1)^2} >
\frac{ \log 2 \, 2^x}{\, (2^x -1)^2} =\frac{ \log 2 \,}{\, (2^x
-1)\, (1- 2^{-x})}.$$

The above means, for  all $x\in [0,1]$,

$$ v(x)=\log(3/2) \, (2^x - 2 + 2^{-x}) > \log 2 \,(\, (3/2) ^x - 2 +
(3/2)^{-x})=u(x).$$

Note that $v(0)=0=u(0).$

As $$v'(x) =  \log(3/2) \,\log 2\, 2^x -  \log(3/2) \,\log 2\,
2^{-x},$$ and
$$u'(x) =  \log(3/2) \,\log 2\, (3/2)^x -  \log(3/2)
\,\log 2\, (3/2)^{-x}.$$

For any non-negative $x$, we have $v'(x)> u'(x)$ and the claim
follows. Therefore,

$$x\mapsto  \frac{1}{((3/2) ^x -1)}\, -  \frac{1}{(2^x-1)} $$
is decreasing on $ [0,1]$. This means that

$$x\mapsto  \frac{1}{\frac{1}{((3/2) ^x -1)}\, -  \frac{1}{(2^x-1)}+1} $$
is  increasing.

Moreover, close to 1 we have $\disp a(x)=(x-1)(2\ln
2-4\ln^2(2)-6\ln\frac32)+O((x-1)^2)$.

Close to 0 we have $\disp
a(x)=\frac{\frac1{\ln\frac32}-\frac1{\ln2}}{x^2\ln2}+C.\frac1{x}+O(1)$.

\end{proof}

From the lemma above we get the property that each positive integer
value of $a$ can be reached. In this way, several renormalization
operators, with different values $a\in \mathbb{N}$, can be considered
in our future reasoning. For each such value $a$, we have the
corresponding values $\alpha_a$ and $\beta_a$. We point out,
however, that it also has meaning to consider real values of $a$
(any positive real is possible) in several of our results (which are
not related to the renormalization operator for the shift)

In the following, we however prefer to keep $\beta$ as parameter, because it gives the length of the interval where the dynamic works :

\begin{proposition}\label{prop-theta-onto}
With our new conditions on $\alpha$, $\beta$ and $a$, the map
$\theta_{\alpha,\beta,a}$, from $\S_{2}^+$ into the interval
$[0,2^\beta-1]$ is onto.
\end{proposition}

\begin{proof}
For a given $x$ in $[0,2^\beta-1]$, we want to exhibit a sequence
$n_{0},n_{1},n_{2},\ldots$ of integers, possibly equal to $+\8$ (in
that case the sequence stops), all positive, except $n_{0}$ which is
non-negative, such that {\footnotesize
$$ x= \disp\frac{  1}{  \disp \frac{(n_0+1)^\beta}{(n_0+2)^\beta - (n_0+1)^\beta}   + \disp\frac{1}{ \disp\frac{(n_1+a)^\alpha}{(n_1+a+1)^\alpha - (n_1+a)^\alpha} + \disp\frac{1}{\disp\frac{n_2^\beta}{(n_2+1)^\beta - n_2^\beta} +   \disp    \frac{1}{\disp \frac{(n_3+a)^\alpha}{(n_3+a+1)^\alpha - (n_3+a)^\alpha } +\frac{1}\ddots }    } }},
$$}

If such a sequence exists, then $\disp
\frac1x=g_{\beta}(n_{0}+1)+r_{0}(x)$. We may, for instance choose
$n_{0}+1$ as the integer $n_{\frac1x}$ (see \eqref{equ1-gauss-alpha}
but for $g_{\beta}$). In that way, \eqref{equ1-alpha-beta} yields
that  $r_{0}(x)\leq (1+\frac1{a+1})^\alpha-1$.  Moreover, $x\le
2^\beta-1$ yields $\frac1x\ge g_{\beta}(1)$, hence $n_{0}\ge 0$. We
thus have now to iterate this construction by induction. Clearly
$\frac{1}{r_{0}(x)}\ge g_{\alpha}(1+a)$, and we can find some
$n_{1}\geq 1$ such that
$$\frac1{r_{0}(x)}=g_{\alpha}(n_{1}+a)+r_{1}(x).$$
We then have $r_{1}(x)\leq \frac1\alpha=2^\beta-1$, and we can
iterate this process.
\end{proof}

We now explain what are the points with finite fractional expansion. In the construction the process stops if and only if some rest $r_{i}(x)$ equals 0. If $i$ is even, we are dealing with the maps $g_{\beta}$ and $\phi_{\beta}$. The symbolic representation of $x$ is  $\bar{x}=  \displaystyle
(\underbrace{0,\ldots,0}_{n_0},\underbrace{1,\ldots,1}_{n_1},\ldots,\underbrace{0,\ldots,0}_{n_i},1,\ldots,1,\ldots)$, and ultimately equals 1.
Just at the right side of $x$, in $x+0$, points have one zero less in their symbolic representation in $\S$ for the $i^{th}$-block, then one 1 and arbitrarily long string of 0's. Just at the left side of $x$, in $x-0$, points have $n_{i}$ 0's and arbitrarily long string of 1.

If $i$ is odd, we are dealing with maps $g_{alpha}$ and
$\phi_{\beta}$. Things are similar, except we have to exchange left
side with right side, and 0 with 1.


\subsection{An associated global
transformation on $[0,2^\beta-1]$}

Let $\alpha\ge 1$, $\beta\le1$  and $a$ such that
\eqref{equ1-alpha-beta} and \eqref{equ1.2-alpha-beta} are satisfied.
Given $\bar{x}=  \disp
(\underbrace{0,\ldots,0}_{n_0},\underbrace{1,\ldots,1}_{n_1},\underbrace{0,\ldots,0}_{n_2},1,\ldots)\in
\Sigma$ we set

$$ \hat\theta_{\beta} (\bar{x}):=\theta_{\alpha,\beta,a}(\bar x)= \disp{\frac{      1}{   g_{\beta}(n_{0}+1)+ \frac{1}{ \disp g_{\alpha}(n_{1}+a) + \frac{1}{\disp g_{\beta}(n_{2}) +       \frac{1}{ \disp g_{\alpha}(n_{3}+a)+... }    } }}}
$$

In the above expression, we can assume the possibility $n_0=0.$
If one consider in $ \{0,1\}^\mathbb{N}$ the lexicographic  order, then the global transformation $\hat{\theta}_{\beta}$ is non-decreasing in each cylinder $\bar{0}$ and $\bar{1}$.

Note that $\hat{\theta}_{\beta} ( \bar0^\8)= 0$, and $\hat{\theta}
_{\beta}( \bar1^\8)= 2^\beta-1$.  Moreover, \eqref{equ2-alpha-beta}
yields
$$\hat\theta_{\beta}(01^\8)=\frac1{\disp\frac{1}{\left(\frac32\right)^\beta-1}+0}=\frac1{\disp\frac1{2^\beta-1}+\frac1{\disp\frac1{(1+\frac1{a+1})^\alpha-1}+0}}=\hat\theta_{\beta}(010^{\8})$$

We then define the transformation $f_\beta:[0,2^\beta-1]\to [0,2^\beta-1]$ by
$$ f_{\beta} (x) = [\,\hat{\theta}_{\beta} \circ \sigma \circ \hat{\theta}_{\beta}^{-1}\,] (x).$$

Namely, if

$$ x= \disp \frac{ 1}{    \frac{(n_0+1)^\beta}{(n_0+2)^\beta - (n_0+1)^\beta} + \disp\frac{1}{ \frac{(n_1+a)^\alpha}{(n_1+a+1)^\alpha - (n_1+a)^\alpha} + \disp\frac{1}{ \frac{n_2^\beta}{(n_2+1)^\beta - n_2^\beta} +  \disp     \frac{1}{ \frac{(n_3+a)^\alpha}{(n_3+a+1)^\alpha - (n_3+a)^\alpha } +... }    } }  }  ,
$$

and if  $n_{0}>0$, then
$$ f_\beta (x)=
 \disp\frac{      1}{   \frac{n_0^\beta}{(n_0+1)^\beta - n_0^\beta} +
 \disp\frac{1}{ \frac{(n_1+a)^\alpha}{(n_1+a+1)^\alpha - (n_1+a)^\alpha} + \disp\frac{1}{ \frac{n_2^\beta}{(n_2+1)^\beta - n_2^\beta} +  \disp     \frac{1}{ \frac{(n_3+a)^\alpha}{(n_3+a+1)^\alpha - (n_3+a)^\alpha } +... }    } }}.
$$

If $n_{0}=0$, and if $n_{1}>1$, then
$$ f_\beta (x)=
 \disp\frac{      1}{   \frac1{2^\beta-1} +
 \disp\frac{1}{ \frac{(n_1+a-1)^\alpha}{(n_1+a)^\alpha - (n_1+a-1)^\alpha} + \disp\frac{1}{ \frac{n_2^\beta}{(n_2+1)^\beta - n_2^\beta} +  \disp     \frac{1}{ \frac{(n_3+a)^\alpha}{(n_3+a+1)^\alpha - (n_3+a)^\alpha } +... }    } }}.
$$

If $n_{0}=0$ and $n_{1}=1$, then
$$ f_\beta (x)=
 \disp\frac{      1}{   \frac{n_2^\beta}{(n_2+1)^\beta - n_2^\beta} +
 \disp\frac{1}{ \frac{(n_3+a)^\alpha}{(n_3+a+1)^\alpha - (n_3+a)^\alpha} + \disp\frac{1}{ \frac{n_4^\beta}{(n_4+1)^\beta - n_4^\beta} +  \disp     \frac{1}{ \frac{(n_5+a)^\alpha}{(n_5+a+1)^\alpha - (n_5+a)^\alpha } +... }    } }}.
$$

When $\beta=1$, the transformation $f_\beta$ plays a similar role of
$F$ in \cite{CLR} page 3540.

Note that $f_\beta(0)=0$, and $f_\beta(1)=1$.

\begin{proposition}\label{prop-deriv-fbeta}
The map $f_\beta$ is increasing and differentiable in each interval
$\left[0, {\left(\frac32\right)^\beta-1} \right)$ and
$\left({\left(\frac32\right)^\beta-1},1\right]$.
\end{proposition}
\begin{proof}
First consider  $x\in \left[0, {\left(\frac32\right)^\beta-1}
\right)$. If $x$ does not have a finite continuous expansion, then $
f_\beta^{'} (x) = \frac{f_\beta(x)^2}{x^2}.$

Indeed, suppose $x=\hat{\theta}_\beta(\bar{x})=\disp
(\underbrace{0,\ldots,0}_{n_0},\underbrace{1,\ldots,1}_{n_1},\underbrace{0,\ldots,0}_{n_2},1,\ldots)\in
\Sigma$, and take $y=\hat{\theta}_\beta(\bar{y})$, with $y$ close to
$x$ in $[0,\left(\frac32\right)^\beta-1)$. Note that $n_{0}>0$.

As $a$ is an integer, discontinuities of the fractional expansion
only appear for a fixed countable set of points (whatever the
``level'' they appear). Assume $x$ is not such a point. Then,
$\bar{x}$ and $\bar{y}$ are close in $\Sigma$. Using the above
expansion suppose $x= \frac{1}{u+b}$ and $y=\frac{1}{u +b^{'}}$,
where $u=g_{\beta}(n_{0}+1)$.

Therefore, $f_\beta(x)= \frac{1}{\zeta+b}$ and
$f_\beta(y)=\frac{1}{\zeta +b^{'}}$. Finally,
$$ \frac{f_\beta(x) - f_\beta(y) }{x-y}= \frac{   \frac{1}{\zeta+b} -  \frac{1}{\zeta +b^{'}}      }{     \frac{1}{u+b}-        \frac{1}{u +b^{'}}}  = \frac{(u+b)\, (u + b^{'} )     }{(\zeta + b) \, (\zeta + b^{'})     }.$$

When $y\to x$ we get $b^{ '} \to b.$ Then, we get the expression
$$ f_\beta^{'} (x) = \frac{f_\beta(x)^2}{x^2}.$$

We claim that this also holds if $x$ is a discontinuity point, but
if $y$ goes to $x+0$ or $x-0$ (depending if the tail of $\bar x$ is
only 0 or only 1). In that case we just have a left or right
derivative.

Let us now study the other side. Note that we only consider the case
$n_{0}>1$ because the map is not continuous in $\disp
{\left(\frac32\right)^\beta-1}$. We have
$$ \frac{f_\beta(x) - f_\beta(y) }{x-y}= \frac{   \frac{1}{\zeta+b} -  \frac{1}{\zeta' +b^{'}}      }{     \frac{1}{u+b}-        \frac{1}{u' +b^{'}}}  = \frac{(u+b)\, (u' + b^{'} )     }{(\zeta + b) \, (\zeta' + b^{'})     },$$
where $u'\neq u$ and $\zeta'\neq\zeta$ if and only if $n_{1}=+\8$.
When $y\rightarrow x$, we do not claim that $u'\rightarrow u$ and
$b'\rightarrow b$, but  the fact that $\hat\theta_{\beta}$ is onto
yields that in $\R$, $u'+b'\rightarrow u+b$ and
$\zeta'+b'\rightarrow \zeta+b$.

Therefore, for any $x$ in $\left[0, {\left(\frac32\right)^\beta-1}
\right)$,
$$f_\beta^{'} (x) = \frac{f_\beta(x)^2}{x^2}.$$

The case $x$ in $\left] {\left(\frac32\right)^\beta-1},1 \right]$ is
similar. Note that this interval is also
$$\left]\frac1{\disp\frac1{2^\beta-1}+{{(1+\frac1{a+1})^\alpha-1}}},1\right],$$
and we have to exchange the variable $x$ with $1-x$.
\end{proof}

An easy calculation shows that on $\left[0,
{\left(\frac32\right)^\beta-1} \right)$, $f_\beta (x)$ is on  the
form $f_\beta(x) = \frac{x}{1+ cx}$. To find $c$ we use the boundary
condition $f_\beta( \left(\frac32\right)^\beta-1)=2^\beta-1$. Then
\eqref{equ1-alpha-beta} yields
$$-c=\frac{1}{\left(\frac32\right)^\beta-1}-\frac1{2^\beta-1}=(1+\frac1{a+1})^\alpha-1.$$

When $a=0$ we get $c=-1$.

In a similar way, we can find a value $d$ such that $f_\beta(x) =
\frac{d\,+ \,(1-d) \,x}{1\,+\, d\, -\,d\,x}$, for $x\in \left]
{\left(\frac32\right)^\beta-1},1 \right]$.

In this way we obtain $f_\beta$ as a new kind of
Manneville-Pomeau-like map. The values of $c$ and $d$ depend on $a$,
$\alpha $ and $\beta$. Note that the jet in each of the two fixed
points of this map is associated to a different parameter, namely,
depend respectively on $c $ and $d$.

\begin{remark}\label{ }
As we said in the introduction, the condition $a\in\N$ does not seem
to be necessary. We can define $\theta_{\alpha,\beta,a}$ with
$a\in\R_{+}$, which give us continuous $\alpha(a)$ and $\beta(a)$.
We however recall that $a\in\N$ was needed to define the
renormalization in the shift.
\end{remark}

\end{document}